\documentclass[12pt]{amsart}
\usepackage{a4}
\usepackage{mathrsfs}
\usepackage{color}
\usepackage{enumitem}
\usepackage{graphicx}
\usepackage{subfigure}
\usepackage[colorlinks=true,
linkcolor=webgreen,
filecolor=webgreen,
citecolor=webgreen]{hyperref}
\definecolor{webgreen}{rgb}{0,0,1}%{1,0.0,.6}
\definecolor{recrown}{rgb}{1,.2,.6}

\begin{document}
\newtheorem{theorem}{Theorem}
\newtheorem{corollary}[theorem]{Corollary}
\newtheorem{lemma}[theorem]{Lemma}
\theoremstyle{definition}
\newtheorem{example}{Example}
\newtheorem{examples}{Examples}
\newtheorem*{notation}{Notation}
\theoremstyle{remark}
\newtheorem*{remarks}{\bf Remarks}
\theoremstyle{thmx}
\newtheorem{thmx}{\bf Theorem}
\renewcommand{\thethmx}{\text{\Alph{thmx}}}% "letter-numbered" theorems
\newtheorem{lemmax}{Lemma}
\renewcommand{\thelemmax}{\text{\Alph{lemmax}}}% "
\theoremstyle{thmx}
\newtheorem{property}{\bf Property}
\renewcommand{\theproperty}{\text{\Alph{property}}}% "letter-numbered" theorems
%\leftmargin=.5in
%\rightmargin=0.5in
%\textwidth=6truein
%\textheight=11.6truein
%\hoffset=-0cm
%\voffset=+-2cm
\theoremstyle{definition}
\newtheorem*{definition}{Definition}
\numberwithin{theorem}{section}
\numberwithin{example}{section}
\newtheorem{remark}{\bf Remark}
\newcommand{\C}{\mathcal{C}_{x,y}}
\newcommand{\s}{\mathbb{S}}
\title[]{\bf Round and sleek subspaces of linear metric spaces and metric spaces}
\markright{}
%\footnotetext{
%}
\author{Jitender Singh$^{\dagger,*}$}
\address{$~^\dagger$ Department of Mathematics, Guru Nanak Dev University, Amritsar-143005, India}
\author{T. D. Narang$^\ddagger$}
\address{$~^\ddagger$ Department of Mathematics, Guru Nanak Dev University, Amritsar-143005, India}
\footnotetext[2]{$^{,*}$Corresponding author: jitender.math@gndu.ac.in}
\footnotetext[3]{tdnarang1948@yahoo.co.in}
%\markright{}
%\parskip=10pt
\date{}
\maketitle
\begin{abstract}
In the recent work [Metrically round and sleek metric spaces, \emph{The Journal of Analysis} (2022), pp 1--17], the authors proved some results on metrically round and sleek linear metric spaces and metric spaces. In continuation, the present article discusses more results on such spaces along with identification of round and sleek subsets of  linear metric spaces and metric spaces in the subspace topology.
%Mathematics Subject Classification (2020): 54E35; 30L99; 46A55; 52A07; 46B20.
\end{abstract}
\section{Introduction}
Let $(X,d)$ be a metric space having at least two points and having no isolated point. For any point $x$ in $X$ and a positive real number $r$, let $B_d(x,r)$ and $B_d[x,r]$ respectively denote the open and the closed balls centered at $x$ and radius $r$ in $X$.

Art\'emiadis \cite{artemiadis} and Wong \cite{wong} discussed those metric spaces in which closure of every open ball is the corresponding closed ball, that is, $\bar{B}_d(x,r)=B_d[x,r]$ for all $x\in X$ and $r>0$. Nathanson  \cite{Na} called such a metric $d$ as a \emph{round metric} for $X$.  A round metric space was then defined to be the one whose topology is induced by a round metric. Among other results in \cite{vasilev}, Vasil'ev proved in particular that a linear metric space  is strictly convex if and only if it has strict ball convexity and is round.

Analogously, Kiventidis \cite{TH} discussed those metric spaces in which interior of every closed ball is the corresponding open ball, that is, $B_d^{\circ}[x,r]=B_d(x,r)$ for all $x\in X$ and $r>0$. Singh and Narang \cite{JSTD2020} called such a metric $d$ as a \emph{sleek metric} for $X$.   A sleek metric space was then defined to be the one whose topology is induced by a sleek metric. Whereas all normed linear spaces are round as well as sleek with respect to the metric induced by the underlying norm, it is not so in the case of linear metric spaces and metric spaces (see \cite{artemiadis,wong,vasilev,TH,JSTD2020}).

Although the two concepts of roundness and sleekness are independent of each other, both of these concepts help in characterizing strictly convex linear metric spaces (see \cite{vasilev,JSTD2020}).  Infact from the characterization results in \cite{vasilev,JSTD2020} regarding strictly convex linear metric spaces via roundness or sleekness and the strict ball convexity, it follows that the classes of round linear metric spaces and sleek linear metric spaces lie between the class of strictly convex linear metric spaces and the class of linear metric spaces. On the other hand, the class of round metrizable spaces lies between complete convex metric spaces and metric spaces, and that of  sleek metrizable spaces lie between externally convex complete or externally convex round metric spaces and metric spaces. These implications can be understood from Fig.~\ref{F1} in which LMS stands for linear metric space and MS stands for metric space, and each arrowhead denotes the direction of the implication ($\subset$) of containment of the corresponding class of spaces.
\begin{figure}[h!!!]
\includegraphics[width=0.75\textwidth]{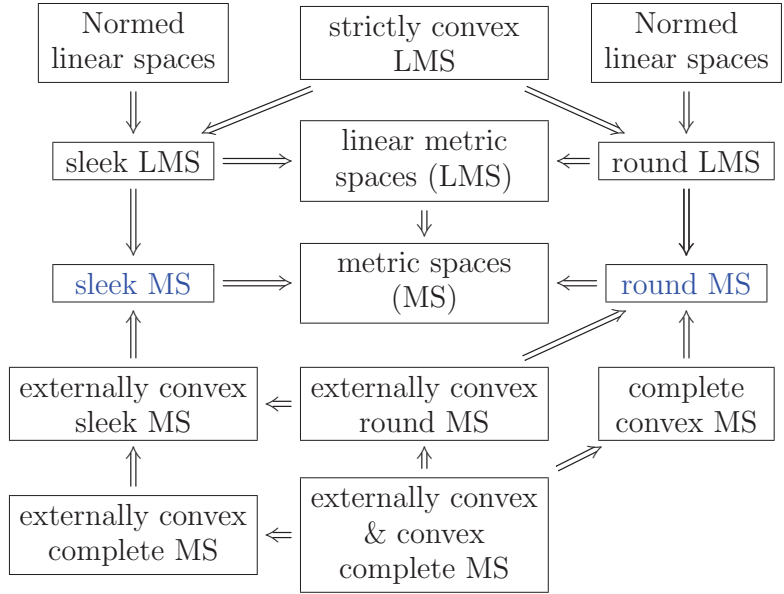}
\caption{The hierarchy of linear metric spaces and metric spaces, based on various conditions.}\label{F1}
\end{figure}

Some of the implications are known in the literature and the others are proved in Section \ref{sec2}.  In continuation to the recent work in \cite{JSTD2020,JSTD2022}, we obtain several interesting features of round and sleek metric spaces with an emphasis on identifying round and sleek subsets of metric spaces in the subspace topology. The paper is organized as follows. Preliminary definitions and results which are used in the subsequent sections are given in Section \ref{sec1}. The main results of this paper are discussed in Section  \ref{sec2}. Some examples are provided in Section \ref{sec3}.
\section{Preliminary definitions and results}\label{sec1}
In this section, we give some definitions and results to be used in the present paper.

A metric space $(X,d)$ is said to be
\begin{enumerate}[label=(\alph*)]%[label=(\Alph*)] [label=(\roman*)]
\item metrically convex  or convex if for every pair of distinct points $x$ and $y$ in $X$, there exists $z\in X\setminus\{x,y\}$ with $d(x,z)+d(z,y)=d(x,y)$. 
\item externally convex if  for every pair of distinct points $x$ and $y$ of $X$, there exists $z\in X\setminus\{x,y\}$ with $d(x,y)+d(y,z)=d(x,z)$.
\end{enumerate}
We call a subset $C$ of a metric space $(X,d)$ as
\begin{enumerate}[resume*]
\item  a convex set if $(C,d)$ is a convex metric space and for every pair of distinct points $x$ and $y$ in $C$, any $z\in X$ satisfying $d(x,z)+d(z,y)=d(x,y)$ belongs to $C$.
\item  an externally convex set if $(C,d)$ is an externally convex metric space and for every pair of distinct points $x$ and $y$ in $C$, any $z\in X$ satisfying $d(x,y)+d(y,z)=d(x,z)$ belongs to $C$.
 \end{enumerate}
 A linear metric space is a topological vector space with a compatible translation invariant metric (see \cite{Rudin}). 
 
 A linear metric space $(X,d)$ is called
   \begin{enumerate}[resume*]
\item strictly convex if $x,y\in B_d[0,r]$, $r>0$ implies $(x+y)/2\in B_d(0,r)$.
\item ball convex if $x,y\in B_d[0,r]$, $r>0$ implies $d(0,(x+y)/2)\leq r$.
\item strictly ball convex$^{\dagger}${\footnotetext[2]{A \emph{strictly ball convex} linear metric space was called \emph{strongly ball convex} in \cite{vasilev}.}} if $x,y\in B_d[0,r]$, $r>0$ implies $(1-t)x+ty\in B_d^{\circ}[0,r]$ for all $t\in (0,1)$.
   \end{enumerate}
    A subset $C$ of a linear metric space is called
    \begin{enumerate}[resume*]
 \item  convex  if $x,y\in C$ implies $(1-t)x+ty\in C$ for all $t\in [0,1]$.
    \end{enumerate}
The main results of this paper rest on the results  stated as under.
\begin{thmx}[\cite{SN79}]\label{thA0}
 If a linear metric space $(X,d)$ is strictly convex, then for every $x\in X\setminus\{0\}$, the map $f_x:[0,\infty]\rightarrow \mathbb{R}$  defined by
 \begin{eqnarray*}
 f_x(t)=d(0,tx),~0\leq t<\infty
 \end{eqnarray*}
 is strictly increasing.
\end{thmx}
\begin{thmx}[\cite{wong}]\label{thA}
A metrically convex complete metric space is round.
\end{thmx}
A connected Riemannian manifold admits a geodesic
curve joining any pair of distinct points on it. The smallest length of the arcs joining such a pair of points is always along the geodesic and is defined by the Riemannian metric tensor (see \cite[p.~21]{J2017}). The resulting metric turns the manifold into a convex metric space. This in view of Theorem \ref{thA} tells us that every compact connected Riemannian manifold is round.
\begin{thmx}[\cite{JSTD2022}]\label{thB} Let $(X,d)$ be a metric space having at least two points.
\begin{enumerate}[label=(\alph*)]
\item\label{Ba} The metric $d$ is not round for $X$ if and only if there exists an open set $U$ in $X$ and $x\in X\setminus U$ for which the map $d(x,\cdot):U\rightarrow \mathbb{R}$ has a minimum.
\item\label{Bb} The metric $d$ is not sleek for $X$ if and only if there exists an open set $U$ in $X$ and $x \in U$ for which the map $d(x,\cdot):U\rightarrow \mathbb{R}$ has a maximum.
\end{enumerate}
\end{thmx}
Theorem \ref{thB} allows an easy  identification of round and sleek metric spaces from the others via attainment of minimum and maximum of the underlying function $d(x,\cdot)$ on an open set. It is easy to observe from this theorem that a non-singleton metrizable space having an isolated point is neither round nor sleek. In this view, the notions of roundness and sleekness make sense only for the non-singleton metrizable spaces which have no isolated point.
\section{Main results}\label{sec2}
It is known (see \cite{TH}) that an open or dense subset of a round (resp. sleek) metric space is round (resp. sleek). The following results  identify a class of round subsets of normed linear spaces and linear metric spaces.
\begin{theorem}\label{th1}
A convex subset of a normed linear space is round in the subspace topology. In particular, a normed linear space is round.
\end{theorem}
\begin{proof}
    Let $C$ be a convex subset of a normed linear space $(X, \|\cdot\|)$. Let $d$ be the metric induced by the norm $\|\cdot\|$ on $X$.  If possible, suppose that $C$ is not round. Then by Theorem \ref{thB}\ref{Ba}, there exist an open subset $O$ of $X$ and a point $x\in C\setminus O$, such that the map $d(x,\cdot):C\cap O\rightarrow \mathbb{R}$ has a minimum value. Consequently, there exists $y\in C\cap O$ for which $0<\|x-y\|\leq \|x-c\|$ for all $c\in C\cap O$. Observe that for all $t\in (0,1)$, we have
    \begin{eqnarray*}
d(x,(1-t)x+ty)=t\|x-y\|<\|x-y\|,
    \end{eqnarray*}
and so, $(1-t)x+ty\not\in C\cap O$. Since $y\in O$ and $O$ is open in $X$, there exists an $\epsilon>0$ such that $B_d(y,\epsilon)\subseteq O$. By Archimedean property of real line, choose a positive integer $n$ for which $\frac{1}{n}\|x-y\|<\epsilon$. Let $z=\frac{x}{n}+\bigl(1-\frac{1}{n}\bigr)y$. Since $C$ is convex and $x,y\in C$, the point $z$ belongs to $C$. We also have
$d(y,z)=\frac{1}{n}\|x-y\|<\epsilon$, and so, $z\in B_d(y,\epsilon)$. Thus, $z\in C\cap B_d(y,\epsilon)\subseteq C\cap O$, where
 \begin{eqnarray*}
\|x-y\|\leq d(x,z)=(1-1/n)\|x-y\|<\|x-y\|,
    \end{eqnarray*}
which is absurd and the theorem holds.
\end{proof}
Whereas a convex subset of a linear metric space need not be round in the subspace topology \cite{JSTD2020}, the following result presents
an analogue  of Theorem \ref{th1} for the strictly convex linear metric spaces.
\begin{theorem}\label{th2}
A convex subset of a strictly convex linear metric space is round in the subspace topology. In particular, a strictly convex linear metric space is round.
\end{theorem}
\begin{proof}
Let $(X,d)$ be a strictly convex linear metric space, and let $C$ be a convex subset of $X$. Assume on the contrary that $C$ is not round. Then by Theorem \ref{thB}\ref{Ba}, there exists an open subset $O$ of $X$ and  $x\in C\setminus O$ such that for some $y\in C\cap O$, $0<d(x,y)\leq d(x,c)$ for all $c\in C\cap O$. Choose $0<\epsilon<d(x,y)$ for which $B_d(y,\epsilon)\subseteq O$. Let $f:[0,1]\rightarrow \mathbb{R}$ be such that
\begin{eqnarray*}
f(t)=d(0,t(x-y)),~t\in [0,1].
\end{eqnarray*}
Clearly, $f$ is continuous with $f(0)=0<\epsilon<f(1)=d(0,x-y)=d(x,y)$. By the intermediate value theorem, there exists $t_\epsilon \in (0,1)$ for which $f(t_\epsilon)=\epsilon/2$, and so, we have $d(0,t_\epsilon(x-y))=f(t_\epsilon)<\epsilon$. Consequently, if we let
\begin{eqnarray*}
z_\epsilon=t_\epsilon x+(1-t_\epsilon)y,
\end{eqnarray*}
then $z_\epsilon\neq x$ and $z_\epsilon \neq y$.
Since $x,y\in C$ and $C$ is convex, we must have $z_\epsilon\in C$. We then have
\begin{eqnarray*}
0<d(y,z_\epsilon)=d(0,t_\epsilon(x-y))<\epsilon,
\end{eqnarray*}
which shows that $z_\epsilon\in C\cap B_d(y,\epsilon)\subseteq C\cap O$. Since $(X,d)$  is strictly convex, by Theorem \ref{thA0} we have
\begin{eqnarray*}
d(x,y)\leq d(x,z_\epsilon)=d(0,(1-t_\epsilon)(x-y))<d(x,y),
\end{eqnarray*}
which is a contradiction, and so, the theorem holds.
\end{proof}
Strict convexity is necessary in the statement of Theorem \ref{th2}. In fact if we take $d(x,y)=\min\{1,|x-y|\}$ for all $x,y\in \mathbb{R}$, then $(\mathbb{R},d)$ is a complete linear metric space having strict ball convexity but $(\mathbb{R},d)$ is not strictly convex, since $d(0,2)=1=d(0,4)$ but $d(0,(2+4)/2)=1$. The metric $d$ is not round for $\mathbb{R}$, since $B_d(0,1)=(-1,1)$, $\bar{B}_d(0,1)=[-1,1]\neq \mathbb{R}=B_d[0,1]$.% (See \cite[Example 1.1]{JSTD2020}).
\begin{corollary}
If $C$ is a convex subset of a strictly convex linear metric space $(X,d)$ and $C\subseteq Y\subseteq \bar{C}$, then $d$ is a round metric for $Y$.
\end{corollary}
\begin{proof} For any $t\in (0,1)$, the map
$f_t:X\times X\rightarrow X$ where $f_t(x,y)=(1-t)x+ty$ is continuous. Since $C$ is convex, we must have $f_t(C\times C)=(1-t)C+tC\subseteq C$ for all $t\in (0,1)$. This, along with the continuity of $f_t$ gives
\begin{eqnarray*}
(1-t)\bar{C}+t\bar{C}=f_t(\bar{C}\times \bar{C})=f_t(\overline{C\times C})\subseteq \overline{f_t(C\times C)}\subseteq  \bar{C},
\end{eqnarray*}
for all $t\in (0,1)$, and so, $\bar{C}$ is convex. By Theorem \ref{th2}, $d$ is a round metric for $\bar{C}$. Since $Y$ is a dense subset of $\bar{C}$ in the subspace topology, the metric $d$ is round for $Y$.
\end{proof}
\begin{corollary}
  Let $C$ be a convex subset of a linear metric space $(X,d)$ having strict ball convexity. If $d$ is a round metric for $X$, then $d$ is a round metric for $C$.
\end{corollary}
\begin{proof} Let $d$ be a round metric for $X$. A round linear metric space having strict ball convexity is strictly convex (See \cite{vasilev}). In view of this and the hypothesis, the linear metric space $(X,d)$ is strictly convex. Consequently,  using Theorem \ref{th2} we conclude that $d$ is a round metric for $C$.
\end{proof}
There exist linear metric spaces $(X,d)$ which do not have strict convexity, but the map $t\mapsto d(0,tx)$, $x\in X\setminus\{0\}$, $t\geq 0$ is still a strictly increasing function of $t$. This can be seen from the following explicit example of \cite{SN81}. Consider the linear metric space $(\mathbb{R}^2,d)$ where
\begin{eqnarray*}
d((x_1,y_1),(x_2,y_2))=\sqrt{|x_1-x_2|+|y_1-y_2|};~(x_1,y_1),(x_2,y_2)\in \mathbb{R}^2.
\end{eqnarray*}
For any point $(x,y)\in\mathbb{R}^2$ with $xy\neq 0$, we have
\begin{eqnarray*}
f_{(x,y)}(t)=d((0,0),t(x,y))=\sqrt{t}\sqrt{|x|+|y|},~t\geq 0,
\end{eqnarray*}
which shows that $f_{(x,y)}$ is a strictly increasing continuous function of $t$. However, the linear metric space $(\mathbb{R}^2,d)$ is not strictly convex, since
\begin{eqnarray*}
d((0,0),(1,0))=d((0,0),(0,1))=1=d((0,0),((1,0)+(0,1))/2).
\end{eqnarray*}
The fact that the space $(\mathbb{R}^2,d)$ is round follows from the following generalization of Theorem \ref{th2}.
\begin{theorem}\label{th4}
Let $C$ be a convex subset of a linear metric space $(X,d)$ having the property that for each $x\in C$ with $x\neq 0$, the map $f_x:[0,1]\rightarrow \mathbb{R}$ where
\begin{eqnarray*}
f_x(t)=d(0,tx)~\text{for all}~ t\in [0,1]
\end{eqnarray*}
is strictly increasing. Then $C$ is a round subset of $X$ in the subspace topology.
\end{theorem}
\begin{proof}
Suppose the contrary that there exists an open subset $O$ of $X$ and $x\in C\setminus O$, $y\in C\cap O$ such that $d(x,y)\leq d(x,u)$ for all $u\in C\cap O$. Then $x\neq y$, and also, there exists $\epsilon>0$ for which $B_d(y,\epsilon)\subseteq O$. For each natural number $n$, if we let $z_n=(1/n)x+(1-1/n)y$, then $z_n\in C$ since $C$ is convex.  By the continuity of $f_{x-y}$ at $0$, there exists a natural number $N$ for which
$f_{x-y}(1/n)<\epsilon$ for all $n\geq N$, where we note that $d(z_n,y)=f_{x-y}(1/n)$. Consequently, $z_n\in C\cap B_d(y,\epsilon)\subseteq C\cap O$, and so, $d(x,y)\leq d(x,z_n)$ for all $n\geq N$. This along with the fact that $f_{x-y}$ is strictly increasing, we get
\begin{eqnarray*}
d(x,y)\leq d(x,z_N)=f_{x-y}(1/N)<f_{x-y}(1)=d(x,y),
\end{eqnarray*}
which is a contradiction.
\end{proof}
\begin{theorem}
An externally convex subset of a normed linear space is sleek in the subspace topology. In particular, a normed linear space is sleek.
\end{theorem}
\begin{proof}
Let $C$ be an externally convex subset of a normed linear space $(X,\|\cdot\|)$. Let $d$ is the metric on $X$ induced by the norm. If possible, suppose that $C$ is not sleek. So, by Theorem \ref{thB}\ref{Bb}, there exists an open subset $O$ of $X$ and $x,y\in C\cap O$, for which $d(x,u)\leq d(x,y)$ for all $u\in C\cap O$. Then $C\cap O\subseteq B_d[x,d(x,y)]$ and $x\neq y$. For  each natural number $n>1$, let $z_n=(1+1/n)y-(1/n)x$. Then we have
 \begin{eqnarray*}
\|x-y\|+\|y-z_n\|=\Bigl(1+\frac{1}{n}\Bigr)\|x-y\|=\|x-z_n\|, ~n\geq 2,
 \end{eqnarray*}
which by  the external convexity of $C$ shows that $z_n\in C$ for all $n\geq 2$. Since $\|x-z_n\|=(1+1/n)\|x-y\|>\|x-y\|$ for all $n$, it follows that $z_n\not\in C\cap O$. Choose $\epsilon>0$ for which $B_d(y,\epsilon)\subseteq O$ so that $C\cap B_d(y,\epsilon)\subseteq C\cap O$. Since $\|z_n-y\|=(1/n)\|x-y\|\rightarrow 0$ as $n\rightarrow \infty$, there exists
a natural number $N$ such that  $\|z_n-y\|<\epsilon$ for all $n\geq N$. In particular,
$z_N\in C\cap B_d(y,\epsilon)\subseteq C\cap O$, which is a contradiction.
\end{proof}
\begin{theorem}\label{STH2}  Let $(X,d)$ be a strictly convex linear metric space. If $C$ is a subset of $X$  having the property that for every pair of distinct points $x$ and $y$ in $C$, there exists a sequence of points $(z_n)$ in $C$ such that
\begin{eqnarray*}
z_n=(1+t_n)y-t_nx
\end{eqnarray*}
for some $t_n\in (0,1)$ for each $n$, and $z_n\rightarrow y$ as $n\rightarrow \infty$, then $d$ is a sleek metric for $C$.
\end{theorem}
\begin{proof}
  Assume the contrary that $d$ is not sleek for $C$. By Theorem \ref{thB}\ref{Bb}, there exists an open subset $O$ of $X$ and $x,y\in C\cap O$, for which $d(x,u)\leq d(x,y)$ for all $u\in C\cap O$. Then $C\cap O\subseteq B_d[x,d(x,y)]$. Since $(X,d)$ is strictly convex, by Theorem \ref{thA0} we have $d(x,z_n)=d(0,(1+t_n)(x-y))>d(x,y)$, and so, $z_n\not\in O$ for all $n$. Let $\epsilon>0$ for which $B_d(y,\epsilon)\subseteq O$ and by the hypothesis choose $n$ large enough such that $d(y,z_n)<\epsilon$. Consequently,  $z_n\in C\cap B_d(y,\epsilon)$, and so $z_n\in C\cap O$. We then have $d(x,y)\geq d(x,z_n)>d(x,y)$,  which is a contradiction.
\end{proof}
\begin{corollary}
Let $(X,d)$ be a linear metric space having strict ball convexity.
Let $C$ be a subset of $X$  having the property that for every pair of distinct points $x$ and $y$ in $C$, there exists a sequence of points $(z_n)$ in $C$ such that
\begin{eqnarray*}
z_n=(1+t_n)y-t_nx
\end{eqnarray*}
for some $t_n\in (0,1)$ for each $n$, and $z_n\rightarrow y$ as $n\rightarrow \infty$. If $d$ is a sleek metric for $X$, then $d$ is a sleek metric for $C$.
\end{corollary}
\begin{proof}
Let $d$ be a sleek metric for $X$. A sleek linear metric space having strict ball convexity is strictly convex (See \cite{JSTD2020}). This along with the hypothesis and Theorem \ref{STH2} yield the desired result.
\end{proof}
We have the following generalization of Theorem \ref{STH2}, whose proof is similar and we omit it.
\begin{theorem}\label{STH3}
Let $C$ be a subset of a linear metric space $(X,d)$ having the following two properties:
\begin{enumerate}[label=(\alph*)]
\item For each $x\in C$ with $x\neq 0$, the map $f_x:[0,1]\rightarrow \mathbb{R}$ where
\begin{eqnarray*}
f_x(t)=d(0,tx)~\text{for all}~ t\in [0,1]
\end{eqnarray*}
is strictly increasing,
\item For every pair of distinct points $x$ and $y$ in $C$, there exists a sequence of points $(z_n)$ of $C$ such that $z_n=(1+t_n)y-t_nx$ for some $t_n\in (0,1)$ for each $n$, and $z_n\rightarrow y$ as $n\rightarrow \infty$.
\end{enumerate}
Then $d$ is a sleek metric for $C$.
\end{theorem}
\begin{remark}
    In view of Theorems \ref{th4} and \ref{STH3}, it follows that a linear metric space $(X,d)$ in which the association $t\mapsto d(0,tx)$, $t\geq 0$ is strictly increasing for all nonzero $x$ in $X$ is round as well as sleek and so is every linear subspace of $(X,d)$.
\end{remark}
Now we discuss some results on roundness and sleekness for metric spaces and identify some round and sleek subsets of metric spaces in the subspace topology.

If $(X,d)$ is  a locally connected metric space, then every connected component of $X$ is open. Consequently, if $d$ is a round (resp. sleek), metric for $X$, then so is $d$ for each connected component of $X$.

In  \cite[Proposition 1]{TH}, Kiventidis proved that a metric $d$ is round for $X$ if for every pair of distinct points $x$ and $y$ in $X$, there exist two sequences $(x_n)$ and $(y_n)$ of points of $(B_d(x,d(x,y))\cap B_d(y,d(x,y)))\setminus \{x,y\}$ such that $x_n\rightarrow x$ and $y_n\rightarrow y$. If in addition, the metric space $(X,d)$ is externally convex then $d$ is a sleek metric for $X$ \cite[Proposition 3]{TH}. The following result is a generalization of  \cite[Proposition 3]{TH}.
\begin{theorem}\label{th3} Let $(X,d)$ be an externally convex metric space. If $d$ is a round metric for $X$, then $d$ is a sleek metric for $X$.
\end{theorem}
\begin{proof}We prove the contrapositive statement.  So, assume that the metric $d$ is not sleek for $X$. We may assume without loss of generality that $X$ has no isolated point. By Theorem \ref{thB}\ref{Bb}, there exists an open set $U$ in $X$ and $x,y\in U$ for which $d(x,u)\leq d(x,y)$ for all $u\in U$. Then $U\subseteq B_d[x,d(x,y)]$. Also, in view of the fact that $X$ has no isolated point gives $x\neq y$.  By the external convexity of $(X,d)$, there exists $z\in X\setminus\{x,y\}$ for which
\begin{eqnarray*}
d(z,y)+d(y,x)=d(z,x)
\end{eqnarray*}
and so $d(x,y)<d(x,z)$. Consequently, $z\not\in B_d[x,d(x,y)]$, and so, $z\not\in U$. Now consider the map $d(z,\cdot):U\rightarrow \mathbb{R}$. If there exists $u\in U$ for which $d(z,u)<d(z,y)$, then this along with the inequality $d(x,u)\leq d(x,y)$ and the triangle inequality yield
\begin{eqnarray*}
d(z,x)\leq d(z,u)+d(u,x)<d(z,y)+d(y,x)=d(z,x),
\end{eqnarray*}
which is absurd. So, $d(z,y)\leq d(z,u)$ for all $u\in U$. Thus, the map $d(z,\cdot):U\rightarrow \mathbb{R}$ attains its minimum value $d(z,y)$ on $U$. By Theorem \ref{thB}\ref{Ba}, the metric $d$ is not round for $X$.
\end{proof}
\begin{corollary}\label{cor2}
 A convex and externally convex complete metric space is sleek.
\end{corollary}
\begin{proof}
Let $(X,d)$ be a convex and externally convex complete metric space.
By Theorem \ref{thA}, $d$ is a round metric for $X$. Since $(X,d)$ is externally convex, by Theorem \ref{th3}, $d$ is a sleek metric for $X$.
\end{proof}
\begin{corollary}\label{cor3}
A convex and (metrically) externally convex subset of a strictly convex linear metric space is sleek in the subspace topology.
\end{corollary}
\begin{proof}
Let $C$ be a convex and externally convex subset of a linear metric space $(X,d)$. By Theorem \ref{th2}, the metric $d$ is round for $C$. Since $(C,d)$ is externally convex, by Theorem \ref{th3} the metric $d$ is sleek for $C$.
\end{proof}
Now we prove the following characterization of round and sleek metrizable spaces via their open subsets.
\begin{theorem}\label{th6}
Let $(X,d)$ be a metric space having no isolated point.
Then $d$ is a round (resp. sleek) metric for $X$ if and only if so is $d$ for each proper nonempty open subset of $X$.
\end{theorem}
\begin{proof} We prove the contrapositive of the statements of the theorem. First we prove the theorem for the case of roundness.

Let there be a proper nonempty open subset $O$ of $X$ such that $d$ is not a round metric for $O$. By Theorem \ref{thB}\ref{Ba}, there exists an open subset $U$ of $X$ with  $x\in O\setminus U$, $y\in O\cap U$ such that $d(x,u)\leq d(x,y)$ for all $u\in O\cap U$. Since $O\cap U$ is an open subset of $X$ such that $x\in X\setminus (O\cap U)$, $y\in O\cap U$, and $d(x,u)\leq d(x,y)$ for all $u\in O\cap U$, by Theorem \ref{thB}\ref{Ba}, the metric $d$ is not round for $X$. Conversely, assume that $d$ is not round for $X$. By Theorem \ref{thB}\ref{Ba}, there exists an open set $U$ in $X$, $x\in X\setminus U$, and $y\in U$ such that $d(x,y)\leq d(x,u)$ for all $u\in U$. Then $U$ is a proper subset of $X$ and since $X$ has no isolated point, there exists an open set $V$ containing $x$ such that $V$ is disjoint from $U$. Since $x$ is not an isolated point of $X$, we can choose a point $z\in V$ with $z\neq x$.  Now define
\begin{eqnarray*}
W=\begin{cases}
U\cup V, &~\text{if}~X\neq U\cup V;\\
U\cup V\setminus\{z\}, &~\text{if}~X=U\cup V.
\end{cases}
\end{eqnarray*}
Then $W$ is a proper nonempty open subset of $X$, and $d$ is not a round metric for $W$, since $U$ is an open subset of $W$ with $x\in W\setminus U$ and $y\in U$ for which $d(x,y)\leq d(x,u)$ for all $u\in U$.

Now we prove the theorem for the case of sleekness. If $d$ is not a sleek metric for $X$, then in view of Theorem \ref{thB}\ref{Bb}, there exists an open set $U$ in $X$ and $x,y\in U$ such that $d(x,u)\leq d(x,y)$ for all $u\in U$. Since $X$ has no isolated point, we can choose a point $z\in U\setminus \{x,y\}$. Let
\begin{eqnarray*}
V=\begin{cases}
U, &~\text{if}~X\neq U;\\
U\setminus\{z\}, &~\text{if}~X=U.
\end{cases}
\end{eqnarray*}
Then $V$ is a proper nonempty open subset of $X$ such that $x,y\in V$ and $d(x,u)\leq d(x,y)$ for all $u\in V$. Consequently, by Theorem \ref{thB}\ref{Bb}, the metric $d$ is not sleek for $V$. Conversely, if the metric $d$ is not sleek for some proper nonempty open subset $V$ of $X$, then by Theorem \ref{thB}\ref{Bb}, there exists an open set $O$ in $X$ and $v\in V\cap O$ for which the map $d(v,\cdot):V\cap O\rightarrow \mathbb{R}$ has a maximum. Observe that $V\cap O$ is an open subset of $X$. So, again by Theorem \ref{thB}\ref{Bb}, the metric $d$ is not sleek for $X$.
\end{proof}
\begin{theorem}\label{th7} Let $X$ be a locally connected metrizable space having a separation. Then $X$ is round (resp. sleek) if and only if
for every pair of distinct connected components $C_1$ and $C_2$ of $X$, the subset $C_1\cup C_2$ of $X$ is a round (resp. sleek) in the subspace topology.
\end{theorem}
\begin{proof}
If $X$ is locally connected metrizable spaces, then each component of $X$ is open and so is the union of any two components. Consequently, if $X$ is round (resp. sleek) then in view of Theorem \ref{th6}, union of any two components is  a round (resp. sleek) in the subspace topology.

We prove the contrapositive of the converse part in the case of roundness. Let $d$ be any metric that induces the topology of $X$, and suppose that $d$ is not a round metric for $X$. By Theorem \ref{thB}\ref{Ba}, there exists an open subset $O$ of $X$, $x\in X\setminus O$ and $y\in O$ such that $d(x,y)\leq d(x,u)$ for all $u\in O$. Since $X$ is locally connected, there exist connected open sets $U$ and $V$ containing $x$ and $y$ respectively, such that $U\subseteq O$ and $V\subseteq O$. Then $U\cup V\subseteq O$. Since $X$ is Hausd\"orff, we can assume without loss of generality that $U$ and $V$ are disjoint. Consequently, if $C_1$ and $C_2$ are the connected components of $X$ containing the points $x$ and $y$ respectively, then $U\subseteq C_1$ and $V\subseteq C_2$. Let $C=C_1\cup C_2$. Then $C$ is open in $X$, and so, $V$ is an open subset of $C$ such that $x\in U=C\setminus V$, $y\in V$ and $d(x,y)\leq d(x,u)$ for all $u\in U$. This in view of Theorem \ref{thB}\ref{Ba} proves that the metric $d$ is not round for $C$. This proves the converse part in the case of roundness.

The converse in the case of sleekness can be proved following the lines of the proof for the case of roundness and so, we omit the proof.
\end{proof}
\begin{remark}
It is worth mentioning that in a locally connected metrizable space $X$, even if every connected component of $X$ is round (resp. sleek), the space $X$ need not be round (resp. sleek). For example, no metric can be round for the locally connected subspace $X=[0,1]\cup [2,3]$ of real line, whereas each of the connected components $[0,1]$ and $[2,3]$ of $X$ is round. The corresponding example for the case of sleekness can be seen in \cite[Example 5]{JSTD2022}.
\end{remark}
\begin{definition}
If the metric space $(X,d)$ is convex, then for any pair of distinct points $x$ and $y$ in $X$ we define  the segment $S(x,y)$ between $x$ and $y$ by
\begin{eqnarray*}
S(x,y)=\{z\in X\setminus\{x,y\}~|~d(x,z)+d(z,y)=d(x,y)\}.
\end{eqnarray*}
Similarly, if the metric space $(X,d)$ is externally convex, then for any pair of distinct points $x$ and $y$ in $X$ we define
\begin{eqnarray*}
S_c(x,y)&=&\{z\in X\setminus\{x,y\}~\mid~d(x,y)+d(y,z)=d(x,z)\}\\
&&\cup\{z'\in X\setminus\{x,y\}~\mid~~d(y,x)+d(x,z')=d(y,z')\}.
\end{eqnarray*}
\end{definition}
A normed linear space $(X,\|\cdot\|)$ is convex as well as externally convex. In such a space if the metric $d$ is induced on $X$ by the norm $\|\cdot\|$, then $S(x,y)\cap B_d(x,r)\neq \emptyset$ and $S_c(x,y)\cap B_d(x,r)\neq \emptyset$ for all pair of distinct points $x$ and $y$ in $X$ and $r>0$. These intersection properties when viewed in convex and externally convex metric spaces yield roundness and sleekness of such metric spaces, respectively.
\begin{theorem}\label{th5} Let $(X,d)$ be a metric space.
 \begin{enumerate}[label=(\alph*)]
\item If $(X,d)$ is convex and has the property that for every pair of distinct points $x$ and $y$ in $X$, the set  $S(x,y)\cap B_d(y,r)$ is nonempty for all $r>0$, then the metric $d$ is round for $X$.
\item If $(X,d)$ is externally convex and has the property that for every pair of distinct points $x$ and $y$ in $X$, the set  $S_c(x,y)\cap B_d(y,r)$ is nonempty for all $r>0$, then the metric $d$ is sleek for $X$.
 \end{enumerate}
\end{theorem}
\begin{proof} We prove the contrapositive of each of the statements of the theorem.

$(a)$ Assume that $d$ is not a round metric for $X$. By Theorem \ref{thB}\ref{Ba}, there exists an open set $U$ in $X$ and $x\in X\setminus U$ with $y\in U$ for which $d(x,y)\leq d(x,u)$ for all $u\in U$. Since $x\neq y$ and by definition  $d(x,z)+d(z,y)=d(x,y)$ for all $z\in S(x,y)$ where $z\neq x,y$, we must have $d(x,z)<d(x,y)$. Consequently, $z\not\in U$, and so, $S(x,y)\cap U=\emptyset$. Since $U$ is open and $y\in U$, there exists $\epsilon>0$ for which $B_d(y,\epsilon)\subseteq U$. It then follows that $S(x,y)\cap B_d(y,\epsilon)=\emptyset$.

$(b)$ Assume that $d$ is not a sleek metric for $X$. By Theorem \ref{thB}\ref{Bb}, there exists an open set $V$ in $X$ and $x,y\in V$ for which $d(x,v)\leq d(x,y)$ for all $v\in V$. Without loss of generality, assume that $X$ has no isolated point. Then $X$ has at least two distinct elements say $x$ and $y$. Also,
\begin{eqnarray*}
d(x,y)+d(y,z)=d(x,z);~d(y,x)+d(x,z')=d(y,z')
\end{eqnarray*}
for  $z,~z'\in S_c(x,y)$, and so, we must have $d(x,y)<\min\{d(x,z),~d(y,z')\}$. Consequently, $S_c(x,y)\cap V=\emptyset$. Since $V$ is open and $y\in V$, there exists $\delta>0$ for which $B_d(y,\delta)\subseteq V$. Thus, $S_c(x,y)\cap B_d(y,\delta)=\emptyset$.
\end{proof}
\section{Examples} \label{sec3}
\begin{example}
The unit circle $\s^1$ is round in the subspace topology of $\mathbb{R}^2$. To see this, let $\gamma:[0,1]\rightarrow \mathbb{R}^2$ be such that
\begin{eqnarray*}
\gamma(t)=(\cos 2\pi t,\sin 2\pi t),~t\in [0,1].
\end{eqnarray*}
Then $\gamma$ is a smooth parametrization of $\s^1$. In fact if we define
\begin{eqnarray*}
d_{\s^1}(\gamma(t),\gamma(s))=\min\Bigl\{\mid t-s\mid,1-\mid t-s\mid\Bigr\},~t,s\in [0,1],
\end{eqnarray*}
then $d_{\s^1}$ is a metric on $\s^1$ inducing the subspace topology, such that  $(\s^1,d_{\s^1})$ is a convex metric space. Also, $\s^1$ is compact, and hence complete. So, by Theorem \ref{thA}, the metric $d_{\s^1}$ is round for $\s^1$. Thus, the unit circle $\s^1$ is a round subspace of $\mathbb{R}^2$.

Since $\s^1$ is compact, no metric on $\s^1$ can be sleek (See \cite[Corollary 5(b)]{JSTD2022}). So, in view of Theorem \ref{th3}, the subspace $\s^1$ is not externally convex with respect to any metric  that induces the topology which $\s^1$ inherits as a subspace of $\mathbb{R}^2$.

More generally, let $\gamma:[a,b]\rightarrow \mathbb{R}^n$ be a smooth curve in $\mathbb{R}^n$, that is, the function $\gamma$ is continuously differentiable on $[a,b]$. Then $\gamma$ is rectifiable and the length $\ell_\gamma(\gamma(t),\gamma(s))$ of $\gamma$ between the points $\gamma(t)$ and $\gamma(s)$, $t,s\in [a,b]$ is defined by
\begin{eqnarray*}\label{e1}
\ell_\gamma(\gamma(t),\gamma(s))=\Bigl|\int_t^s \|\gamma'(u)\|du\Bigr|.
\end{eqnarray*}
Now if $A$ is a smooth surface in $\mathbb{R}^n$ and we define $d_A:A\times A\rightarrow \mathbb{R}$ as the minimum over the lengths of all paths between a given pair of points, that is,
\begin{eqnarray*}
d_A(x,y)=\inf_{\gamma} \{\ell(\gamma)~|~\gamma ~\text{is a smooth curve joining the points }~x~\text{and}~y~\text{in A}\},
\end{eqnarray*}
then $d_A$ is a metric on the subspace $A$, such that $(A,d_A)$ is a convex metric space. If $A$ is compact, then in view of Theorem \ref{thA}, $d_A$ is a round metric for $A$. 
\end{example}
%%%%%%%%%%%%%%%%%%%%%%%%%%%%%%%%%%%%%%%%%%%%%%%%%%%%%
\begin{example}
The set $\mathscr{C}[0,1]$  of all continuous real-valued functions defined on $[0,1]$ equipped with the $\sup$-norm, that is,
%\begin{eqnarray*}
$\|f\|=\sup_{t\in [0,1]}|f(t)|$  for all $f\in \mathscr{C}[0,1]$,
%\end{eqnarray*}
is a Banach space. For any $c\in [0,1]$, let
\begin{eqnarray*}
M_c=\{f\in \mathscr{C}[0,1]~\mid~f(0)=c\}.
\end{eqnarray*}
For any $f,g\in M_c$ and $t\in \mathbb{R}$, clearly $(1-t)f+tg\in \mathscr{C}[0,1]$, and also,
\begin{eqnarray*}
(1-t)f(0)+tg(0)=c.
\end{eqnarray*}
It follows that $(1-t)f+tg\in M_c$ for all $t\in \mathbb{R}$. So, by Theorem \ref{th1}, the subset $M_c$ of $\mathscr{C}[0,1]$ is round in the subspace topology, and by Theorem \ref{STH2} the subset $M_c$ is sleek in the subspace topology.

Also, the subset $D\subset \mathscr{C}[0,1]$ consisting of all $f$ with
\begin{eqnarray*}
\sup_{t\in[0,1]}\mid f(t)\mid\leq 1
\end{eqnarray*}
is convex which by Theorem \ref{th1} is round in the subspace topology.
\end{example}
\begin{example}
The product topology of the countably infinite product $\mathbb{R}^\omega$ of real line with itself is metrizable. In fact the product topology  of $\mathbb{R}^\omega$ is recovered from the following translation invariant metric $D$.
\begin{eqnarray*}
D((x_n),(y_n))=\sum_{n=1}^\infty 2^{-n}\frac{|x_n-y_n|}{1+|x_n-y_n|},~(x_n),(y_n)\in \mathbb{R}^\omega.
\end{eqnarray*}
The linear metric space  $(\mathbb{R}^\omega,D)$ is complete. Also, for any nonzero real number $c$, the map defined by $t\mapsto \frac{t|c|}{1+t|c|}$, $t>0$ is strictly increasing and continuous. Consequently, for $(x_n)\neq 0$ and $t\in [0,1]$,  we have
\begin{eqnarray*}
D(0,t(x_n))=D(0,(tx_n))=\sum_{n=1}^\infty 2^{-n}\frac{t|x_n|}{1+t |x_n|},
\end{eqnarray*}
which shows that the map defined by $t\mapsto D(0,t(x_n))$ is strictly increasing  for $t\in [0,1]$.
By Theorem \ref{th4}, the linear metric space $(\mathbb{R}^\omega,D)$ is round. This conclusion also follows  from a result  of Nathanson \cite{Na} that a countable product of round metrizable spaces is round.

Now let $\mathbb{R}^\infty$ be the set of all those sequences of real numbers each of which is eventually zero. Observe that if  $(x_n)$ and $(y_n)$ are sequences of real numbers each of which is eventually zero, then so is the sequence $((1-t)x_n+ty_n)$ for all $t\in \mathbb{R}$. By Theorem \ref{th4}, the subset $\mathbb{R}^\infty$ of $\mathbb{R}^\omega$ is round in the subspace topology.

Since $\mathbb{R}^\infty$ has no isolated point, by Theorem \ref{STH2} the subset $\mathbb{R}^\infty$ of $\mathbb{R}^\omega$ is sleek in the subspace topology.
\end{example}
\begin{example}
The function space $L^p[0,1]$ for $0<p<1$ is a non-normable complete linear metric space with the translation invariant metric $d$, where
\begin{eqnarray*}
d(f,g) &=& \int_{[0,1]}|f-g|^p,~\text{for all}~f,g\in  L^p[0,1].
\end{eqnarray*}
Observe that for all $t\in [0,1]$ and $0\neq f\in L^p[0,1]$, we have
\begin{eqnarray*}
d(0,tf) &=& t^p \int_{[0,1]}|f|^p,
\end{eqnarray*}
which shows that the map defined by $t\mapsto d(0,tf)$, $0\leq t\leq 1$  is strictly increasing, since so is the map defined by $t\mapsto t^p$. Consequently, by Theorem \ref{th4}, the metric $d$ is round for $L^p[0,1]$ as well as each  convex subset of $L^p[0,1]$.
Here, we note that the linear metric space $(L^p[0,1],d)$ is not ball convex.

Now for $f,g\in L^p[0,1]$, $0<p<1$, if we define $h_n=(1+1/n)g-(1/n)f$ for each positive integer $n>1$, then
\begin{eqnarray*}
d(h_n,g) &=& \Bigl(\frac{1}{n^p} \int_{[0,1]}|f-g|^p\Bigr)\rightarrow 0~\text{as}~n\rightarrow \infty.
\end{eqnarray*}
So, by Theorem \ref{STH3}, the metric $d$ is sleek for $L^p[0,1]$ as well as every linear subspace of $L^p[0,1]$.
\end{example}
\begin{example}
Consider the subset $X=\s^1\cup (\{2\}\times\mathbb{R})$ of $\mathbb{R}^2$ in the subspace topology. Here, the subset $\s^1$ is open in $X$ but never sleek. This in view of Theorem \ref{th6} shows that no metric for $X$ can be sleek.
\end{example}
%\subsection*{Acknowledgments.}
%The authors are thankful to the learned referee for the useful suggestions and constructive criticism in improving the paper.
%The present research is supported by Science and Engineering Research Board(SERB), a
%statutory body of Department of Science and Technology (DST), Govt. of India through the project grant
%no. MTR/2017/000575 awarded to Dr. Jitender Singh under the MATRICS Scheme.
\subsection*{Compliance with Ethical Standards}
The authors declare that they have no conflict of interest.
%\subsection*{Ethical approval} This article does not contain any studies with human participants or animals performed by any of the authors.


\begin{thebibliography}{1}
\bibitem {artemiadis} N.~Art\'emiadis, A remark on metric spaces, \emph{Proc. Kon. Ned. Akad. van Wetensch} A {68}(1965), 316--318.
\bibitem {wong}J.~S.~W.  Wong, Remarks on metric spaces, \emph{Indag.~Math. (Proceedings)} {69} (1966), 70--73.
\bibitem {Na} M.~B. Nathanson, Round metric spaces, \emph{The American Mathematical Monthly}  {82}, no. 7 (1975), 738--741.
\bibitem {vasilev} A.~I. Vasil'ev, Chebychev sets and strong convexity of metric linear spaces, \emph{Math. Notes.} {25} (1979), 335--340. \url{http://mi.mathnet.ru/eng/mz8328}
\bibitem {TH} T. Kiventidis, Some properties of the spheres in metric spaces, \emph{Indian J. pure appl. Math.} {19}, no. 11, (1988), 1077--1080.
\bibitem {JSTD2020} J. Singh and T.~D. Narang, Remarks on balls in metric spaces, \emph{The Journal of Analysis} {29} (2021), 1093--1103. \url{https://doi.org/10.1007/s41478-020-00297-z}
\bibitem {Rudin} W. Rudin, \emph{Functional Analysis}, Tata McGraw-Hill (Indian Ed.)  (2006).
\bibitem{SN79} K.~P.~R. Sastry and S.~V.~R. Naidu, Convexity conditions in metric linear spaces, \emph{Math. Seminar Notes} {7} (1979), 235--251.
\bibitem {J2017} J. Jost, \emph{Riemannian Geometry and Geometric Analysis}, Springer (2017).
\bibitem {JSTD2022} J. Singh and T.~D. Narang, Metrically round and sleek metric spaces, \emph{The Journal of Analysis} %{29} (2021), 1093--1103.
 (2022), pp 1--17. \url{https://doi.org/10.1007/s41478-022-00459-1}
\bibitem{SN81} K.~P.~R. Sastry and S.~V.~R. Naidu, Uniform convexity and strict convexity in metric linear spaces, \emph{Math. Nachr.} {104} (1981), 331--347.
\end{thebibliography}
\end{document}